\documentclass[a4paper,12pt
]{article}
\usepackage{amsmath,amssymb,amsthm}
\usepackage{mathtools}
\usepackage{appendix}
\numberwithin{equation}{section}
\usepackage[margin=20mm]{geometry}

\newtheorem{thm}{Theorem}[section]
\newtheorem{prop}[thm]{Proposition}

\newtheorem{exam}[thm]{Example}
\newtheorem{rem}[thm]{Remark}
\newtheorem{lem}[thm]{Lemma}

\newtheorem{assum}[thm]{Assumption}

\def\E{{\cal E}}
\def\F{{\cal F}}

\def\R{{\mathbb R}}

\def\d{{\rm d}}

\makeatletter

\title{
Essential spectrum for Brox-type diffusion processes
}

\author{Yuichi Shiozawa\thanks{Department of Mathematical Sciences, Faculty of Science and Engineering,
Doshisha University,
1-3, Tatara Miyakodani, Kyotanabe, Kyoto, 610-0394,
Japan; \texttt{yshiozawa@mail.doshisha.ac.jp}}\qquad
Jian Wang\thanks{School of Mathematics and Statistics \& Key Laboratory of Analytical Mathematics and Applications (Ministry of Education) \& Fujian Provincial Key Laboratory of Statistics and Artificial Intelligence,
Fujian Normal University, Fuzhou, 350007, P.R. China; \texttt{jianwang@fjnu.edu.cn}}}

\begin{document}
\maketitle
\begin{abstract}
This paper investigates the essential spectrum and compactness property of Markov semigroups
generated by multi-dimensional Brox
diffusion processes
under two types of random media: stationary Gaussian random fields and semi-selfsimilar L\'evy random fields.
For Gaussian environments satisfying general stationary covariance growth conditions,
we prove
that
the associated semigroup is almost surely noncompact;
if the covariance function grows sublinearly at infinity,
then
the bottom of the essential spectrum vanishes
almost surely,
with fractional Brownian fields as a concrete example.
For multi-dimensional semi-selfsimilar L\'evy environments with scaling index $\alpha\in(1,2)$,
we show
that
the essential spectral bottom equals zero almost surely for arbitrary space dimension $d\ge1$, and the same conclusion holds for one-dimensional symmetric $\alpha$-stable L\'evy processes with any $\alpha\in(0,2)$. Furthermore, we study one-dimensional diffusion operators perturbed by random L\'evy drift. When $0<\delta\le 1/\alpha$, random environmental fluctuations can destroy the compactness of semigroups induced by deterministic power-law potentials $\pm|x|^\delta$, even if the unperturbed semigroup is compact.
The analysis relies on almost sure volume growth estimates for the random reference measure induced by environmental potentials,
sample path asymptotics of random fields, ergodic theory of scaling transforms, and spectral criteria for regular Dirichlet forms.
A unified framework linking sample path
behaviors
of random potentials to spectral characteristics of diffusion semigroups is established,
and several open problems for large potential exponents are stated.

\end{abstract}


\section{Introduction}
Diffusion processes
in random media are a central topic in probability theory and statistical physics.
One of the core objects in this topic is the so-called Brox diffusion processes, initiated by Brox \cite{B86} in one spatial dimension.
In this paper, we are concerned with spectral properties of the multi-dimensional Brox diffusion operator
\[
L^{W}=\frac12\big(\Delta-\nabla W\cdot\nabla\big)
\]
defined on $L^2(\mathbb{R}^d;\mu^W)$ with a random reference measure $\mu^W({\rm d}x)=e^{-W(x)}\,{\rm d}x$,
where $W(x)$ stands for a random environmental potential.
The essential spectrum $\sigma_{\text{ess}}(-L^W)$ determines the compactness of the associated semigroup $\{P_t^W\}_{t>0}$.
In particular, $\{P_t^W\}_{t>0}$ is compact if and only if $\lambda_{\text{ess}}^W:=\inf\sigma_{\text{ess}}(-L^W)=\infty$;
that is, $\{P_t^W\}_{t>0}$ is noncompact if and only if $\lambda_{\text{ess}}^W\in [0,\infty)$.

Classical results focus on deterministic power potentials $w(x)=\pm|x|^\delta$ (here we write $W(x)$ as $w(x)$ in the deterministic setting).
Then, the associated semigroup $\{P_t^w\}_{t>0}$ is compact if and only if $\delta>1$.
Recall that the zero potential (i.e., $w(x)=0$ for all $x\in \R^d$) reduces to $\frac12\Delta$ with fully noncompact semigroup.
These deterministic benchmarks fail to capture real-world disordered systems,
where potentials are fluctuating random fields.
Two mainstream random field models remain understudied for the compactness:
stationary Gaussian fields (including fractional Brownian fields with long-range dependence) and semi-selfsimilar L\'evy fields with scaling jump paths.
Existing literature mainly addresses recurrence/transience (see, e.g., \cite{FNT87, KK20, KTT15, KTT16, KTT17, Tanaka93}),
lacking systematic quenched analysis of
the
essential spectrum and the key open question:
can random drift perturbations break the compactness of confining deterministic diffusion semigroups?

Theoretically, this work builds a universal link among
the
random field sample path behaviors,
the volume growth of
the
random measure $\mu^W$,
and the essential spectrum
for the Dirichlet form.
We develop tailored probabilistic tools for Gaussian extremal estimates and L\'evy scaling ergodicity,
extending deterministic spectral criteria to almost sure quenched settings for random media.
We also resolve a fundamental qualitative spectral problem,
revealing that random fluctuations can completely alter
the
particle confinement encoded by deterministic potentials.

Practically,
diffusion processes in random environment
model particle transport in porous materials \cite{HB}, turbulent flow \cite{AM} and amorphous semiconductors \cite{An}:
on the other hand, the compactness of Markovian semigroups is related to the confinement property of the drift part
(see, e.g., \cite[Section 4.10]{BGL14} or \cite[Section 4.5]{P14}).
Our conclusions quantify how correlated Gaussian disorder or jump-type L\'evy disorder destroys confinement, supporting simulation of heterogeneous media transport. Meanwhile, fractional Brownian and semi-selfsimilar L\'evy fields widely appear in finance, hydrology and risk modeling, making our spectral results applicable to long-term fluctuation analysis of stochastic signals. The main contributions of the paper are as follows:
\begin{enumerate}
    \item \textbf{Spectral noncompactness under Gaussian environment (Theorem \ref{thm:gauss}).}\,\,
    Under general stationary Gaussian covariance assumptions covering fractional Brownian fields,
    we prove
    that
    $\{P_t^W\}_{t>0}$ is almost surely noncompact when
    the
    covariance grows sublinearly at infinity.
    If
    the
    covariance decays sufficiently fast relative to
    the
    spatial radius,
    then
    the bottom of essential spectrum satisfies $\lambda_{\text{ess}}^W=0$ $Q$-a.s.
    We establish three almost sure volume growth properties of $\mu^W$ as key intermediate results (Proposition \ref{prop:volume-gamma}).

    \item \textbf{Semi-selfsimilar L\'evy environments in arbitrary dimension
    (Theorem \ref{thm:semi}).}\,\,
    We construct multi-dimensional random fields from independent semi-selfsimilar L\'evy processes with scaling index $\alpha\in(0,2)$.
    For $\alpha\in(1,2)$, if one has
    positive probability of negative L\'evy increments,
    then
    $\lambda_{\text{ess}}^W=0$ almost surely for all $d\ge1$.
    The main ingredients for the proof are the strong mixing property of the scaling transform and uniform upper bounds of sample paths for semi-selfsimilar L\'evy processes.
    When the random field is given by
    one-dimensional symmetric $\alpha$-stable processes with any $\alpha\in(0,2)$,
    we also establish $\lambda_{\text{ess}}^W=0$ almost surely via divergence of L\'evy exponential functionals
    and the noncompactness criterion of Pinsky \cite{P09} for one-dimensional diffusion operators (Proposition \ref{prop:symm-stable}).

    \item \textbf{Random drift destroys semigroup compactness of confining potentials
    (Proposition \ref{prop:drift}).}
    In one dimension, we perturb compact-generating power potentials $|x|^\delta$ and $-|x|^\delta$ by
    the
    L\'evy random drift.
    When $0<\delta\le1/\alpha$, the perturbed generator has zero essential spectral bottom and so generates a noncompact semigroup.
      Notably,
    if $\alpha\in(0,1)$ and thus  $1/\alpha>1$, then we can take $\delta\in (1,1/\alpha)$ so that  the original semigroup is compact
    but the perturbed one is noncompact. This proves that random fluctuations can reverse deterministic confinement properties.

\end{enumerate}

For both Gaussian and semi-selfsimilar L\'evy models,
we follow a consistent framework:
(a) to
analyze random field sample paths to derive
the
volume growth of $\mu^W$;
(b) to
apply the spectral criteria via
the Dirichlet form;
(c) to
translate spectral results to the compactness of
the
Markov semigroup. We believe that such approach can be generalized to other random potential models.

\ \

The rest of this paper is organized as follows.
Section \ref{sect:prelim} provides preliminaries on Dirichlet forms and spectral criteria.
Section \ref{sect:gauss} treats Gaussian environments, Section \ref{e:section2.2} considers semi-selfsimilar L\'evy media,
and Section \ref{section:per} discusses perturbation effects of
the
random drift on deterministic diffusion operators.

\section{Preliminaries}\label{sect:prelim}
In this section, we first introduce our setting which is based on the Dirichlet form theory.
We then recall a relation between the compactness property of Markovian semigroups with
the spectral property of generators.

Let ${\rm d}x$ be the Lebesgue measure on ${\mathbb R}^d$,
and $w: \R^d \to \R$ a locally bounded measurable function.
Let $\mu^w$ be a measure on ${\mathbb R}^d$ defined by $\mu^w({\rm d}x)=e^{-w(x)}\,{\rm d}x$,
which is a positive Radon measure on ${\mathbb R}^d$ with full support.
We now define the quadratic form $({\cal E}^w,{\cal D}({\cal E}^w))$ on $L^2({\mathbb R}^d;\mu^w)$ as follows:
\begin{equation*}
\begin{split}
{\cal D}({\cal E}^w)
&=\left\{u\in L^2({\mathbb R}^d;\mu^w): \int_{{\mathbb R}^d}|\nabla u(x)|^2 e^{-w(x)}\,{\rm d}x<\infty\right\},\\
{\cal E}^w(u,v)
&=\frac{1}{2}\int_{{\mathbb R}^d}\nabla u(x)\cdot \nabla v(x) e^{-w(x)}\,{\rm d}x, \quad u,v\in {\cal D}({\cal E}^w).
\end{split}
\end{equation*}
We can then induce the Banach norm $\|\cdot\|_w$ on ${\cal D}({\cal E}^w)$ by
\[
\|u\|_w=\left({\cal E}^w(u,u)+\|u\|_{L^2(\R^d;\mu^w)}^2\right)^{1/2}, \quad u\in {\cal D}({\cal E}^w).
\]
Let $C_c^{\infty}({\mathbb R}^d)$ denote the set of smooth functions with compact support in ${\mathbb R}^d$.
Then $C_c^{\infty}({\mathbb R}^d)\subset {\cal D}({\cal E}^w)$
and $({\cal E}^w,C_c^{\infty}({\mathbb R}^d))$ is a symmetric Markovian closable form
(see, e.g.,  \cite[Section 1.1]{FOT11} for definition).
In fact, the Markovian property follows from a similar argument as in \cite[Example 1.2.1]{FOT11},
and the closability is proved in \cite[Lemma I.1]{M94}.
Hence, if we write $({\cal E}^w,{\cal F}^w)$ for the closure of $({\cal E}^w,C_c^{\infty}({\mathbb R}^d))$ with respect to the norm $\|\cdot\|_w$,
then $({\cal E}^w,{\cal F}^w)$ is a regular Dirichlet form on $L^2({\mathbb R}^d;\mu^w)$.

Let $L^w$ be a non-positive definite self-adjoint operator on $L^2(\R^d;\mu^w)$ such that
its domain $D(L^w)$ satisfies $D(L^w)\subset {\cal F}^w$ and
\[
{\cal E}^w(u,v)=\langle-L^w u,v\rangle_{L^2(\R^d;\mu^w)}, \quad u\in D(L^w), \ v\in {\cal F}^w
\]
(\cite[Theorem 1.3.1 and Corollary 1.3.1]{FOT11}).
Then $L^w$ has the following formal expression:
\[
L^w=\frac{1}{2}(\Delta-\nabla w\cdot \nabla)=\frac{1}{2}e^{w}\sum_{k=1}^d\frac{\partial}{\partial x_k}\left(e^{-w}\frac{\partial }{\partial x_k}\right).
\]
Let $\sigma_{{\rm ess}}(-L^w)$ be the set of the essential spectrum of the operator $-L^w$,
and let $\lambda_{{\rm ess}}^w=\inf\sigma_{{\rm ess}}(-L^w)$ denote the bottom of the essential spectrum of $-L^w$.
If we write $\{P_t^w\}_{t>0}$ for the Markovian semigroup on $L^2(\R^d;\mu^w)$
generated by $L^w$,
then $\{P_t^w\}_{t>0}$ is compact if and only if $\sigma_{{\rm ess}}(-L^w)=\emptyset$,
that is, $\lambda_{{\rm ess}}^w=\infty$ (see, e.g., \cite[Theorem 0.3.9]{Wa05}).

If we take $w=0$, then $\mu^0({\rm d}x)={\rm d}x$ and the generator $L^0$ is nothing but the half of the Laplace operator on $\R^d$.
In particular, $\sigma_{{\rm ess}}(-L^0)=[0,\infty)$
and the corresponding Markovian semigroup
on $L^2(\R^d;{\rm d}x)$ is noncompact.
On the other hand,
if $w\ne 0$, then the noncompactness property of the Markovian semigroup
depends on the behavior of the function $w$ at infinity.
For instance, by using \cite[Corollary 1]{P09}, we can verify the assertion as follows: For $\delta>0$, let $w(x)=|x|^{\delta}$ or $w(x)=-|x|^{\delta}$ for $x\in \R^d$.
Then the Markovian semigroup $\{P_t^w\}_{t>0}$ on $L^2(\R^d; \mu^w)$ is noncompact
if and only if $0<\delta\le1$.
In particular, if $\delta=1$, then $\lambda_{{\rm ess}}^w\in (0,\infty)$.

\section{Gaussian environment}\label{sect:gauss}
In this section, we show the noncompactness of the Markovian
semigroup
associated with the Brox diffusion process in the Gaussian environment.
Let $(\{W(x)\}_{x\in {\mathbb R}^d},Q)$ be a real-valued continuous centered Gaussian process with index set $\R^d$.
Following the formulation of K\^ono \cite{K75},
we make the next  assumption on the covariance structure of $(\{W(x)\}_{x\in {\mathbb R}^d},Q)$.
In what follows, we use the notation $E$ for the expectation with respect to the probability measure $Q$.

\begin{assum}\label{assum:1}\rm
\begin{enumerate}
\item[{\rm (1)}]
The covariance structure of $(\{W(x)\}_{x\in {\mathbb R}^d},Q)$ is stationary, that is,
there exists a function $\gamma:{\mathbb R}^d \to [0,\infty)$ such that
$\gamma(0)=0$, $\gamma(x)>0$ for $x\ne 0$, and
\[
E[(W(x)-W(y))^2]=\gamma(x-y)^2, \quad x,y\in {\mathbb R}^d.
\]
Moreover, there exists a family of nondecreasing functions $\sigma_i: [0,\infty) \to [0,\infty)$, $1\le i\le l$,
such that,
for some positive integers $d_1,\dots,d_l$ with $d_1+\dots+d_l=d$,
\[
\gamma(x)^2=\sum_{i=1}^l \sigma_i(|x_i|)^2, \quad x=(x_1,\dots, x_l)\in {\mathbb R}^{d_1}\times \cdots\times {\mathbb R}^{d_l}.
\]
\item[{\rm (2)}]
The functions $\gamma$ and $\sigma_i$,
$1\le i\le l$,  in {\rm (1)} satisfy the following conditions.
\begin{enumerate}
\item The function $\gamma(x)$ satisfies
\begin{equation}\label{eq:gamma-1}
\lim_{|x|\rightarrow\infty}\frac{\gamma(x) \sqrt{\log\log |x|}}{\log |x|}=\infty.
\end{equation}

\item For every $1\le i\le l$, the function $\sigma_i$ is non-decreasing and continuously differentiable on $[0,\infty)$.
In particular, there exist constants $\alpha_i\in (0,1)$ and $\beta_i\in (0,1)$
such that $\sigma_i$ is regularly varying at $x=0$ with rate $\alpha_i$, and at $x=\infty$ with rate $\beta_i$, respectively.

\item There exists a constant $c_0>0$ such that for every $1\le i\le l$,  
$x|\sigma_i'(x)|\le c_0\sigma_i(x)$ for any $x> 0$.
\end{enumerate}
\end{enumerate}
\end{assum}

For each $Q$-a.s.\ realization of $\{W(x)\}_{x\in {\mathbb R}^d}$,
$W$ is a continuous function on $\R^d$ and so locally bounded.
Hence we can introduce a regular Dirichlet form $(\E^W,\F^W)$ on $L^2(\R^d,\mu^W)$ as in
Section \ref{sect:prelim}.
Let $\{P_t^W\}_{t>0}$ be the Markovian semigroup on $L^2({\mathbb R}^d;\mu^W)$
associated with $({\cal E}^W,{\cal F}^W)$,
and let $\lambda_{{\rm ess}}^W$ be the bottom of the essential spectrum of $-L^W$.
We then have
\begin{thm}\label{thm:gauss}
Let Assumptions {\rm \ref{assum:1}}  be fulfilled.
If
\begin{equation}\label{eq:gamma-2}
\limsup_{|x|\rightarrow\infty}\frac{\gamma(x)\sqrt{\log\log |x|}}{|x|}<\infty,
\end{equation}
then $\{P_t^W\}_{t>0}$ is
noncompact, $Q$-a.s.
In particular, if there exists some $p\in (0,1)$ such that
\begin{equation}\label{eq:gamma-3}
\limsup_{|x|\rightarrow\infty}\frac{\gamma(x)\sqrt{\log\log |x|}}{|x|^p}=0,
\end{equation}
then $\lambda_{{\rm ess}}^W=0$, $Q$-a.s.
\end{thm}

We will prove Theorem \ref{thm:gauss}
by way of the volume growth upper bound
for the bottom of the essential spectrum in \cite[Theorem 1]{N98} (see also \cite[Theorem 3.2]{S25}).
In order to do so, we reveal the almost sure volume growth behavior of $\mu^W$.
For $x\in {\mathbb R}^d$ and $r>0$, let
\[
B(x,r)=\{y\in {\mathbb R}^d : |y-x|\le r\}
\]
be a closed Euclidean ball with center $x$ and radius $r$.
\begin{prop}\label{prop:volume-gamma}
Let Assumption {\rm \ref{assum:1}}  be fulfilled.
Then the next three assertions hold.
\begin{enumerate}
\item[{\rm (1)}] $\lim_{R\rightarrow\infty}\mu^W(B(0,R))=\infty$, $Q$-a.s.
\item[{\rm (2)}]
If \eqref{eq:gamma-2}   holds,
then
$$
\limsup_{R\rightarrow\infty}\frac{1}{R}\log \left(\mu^W(B(0,R))\right)<\infty, \quad \text{$Q$-a.s.}
$$
\item[{\rm (3)}]
If \eqref{eq:gamma-3} holds for some $p\in (0,1)$, then
\[
\lim_{R\rightarrow\infty}\frac{1}{R}\log  \left(\mu^W(B(0,R))\right)=0, \quad \text{$Q$-a.s.}
\]
\end{enumerate}
\end{prop}

\begin{proof}
We first prove (1).
Let
\[
\mathcal B=\{f:B(0,1) \to {\mathbb R}\, \big|\,\text{continuous and $f(0)=0$}\}, \quad \|f\|=\sup_{y\in B(0,1)}|f(y)|.
\]
Then, $(\mathcal B,\|\cdot\|)$ is a separable Banach space.
Hence, Fernique's theorem (\cite[Exercise 6.1 and Corollary 6.1]{L12}) implies
that for some $\alpha>0$,
\[
E\left[\exp\left(\alpha \|W\|^2\right)\right]=E\left[\exp\left(\alpha \sup_{y\in B(0,1)}|W(y)|^2\right)\right]<\infty.
\]
Since $\{W(x)\}_{x\in {\mathbb R}^d}$ is a stationary process by Assumption {\rm \ref{assum:1}},
we see that for any $q\in {\mathbb R}^d$,
$\sup_{y\in B(q,1)}|W(y)-W(q)|$ and $\sup_{y\in B(0,1)}|W(y)|$ are identical in distribution.
Combining this with Chebyshev's inequality,
we have for any $a>0$ and $q\in \R^d$,
\begin{equation}\label{eq:c-ineq}
\begin{split}
P\left(\sup_{y\in B(q,1)}|W(y)-W(q)|\ge a\right)
&=P\left(\sup_{y\in B(0,1)}|W(y)|\ge a\right)\\
&\le \frac{1}{e^{\alpha a^2}}E\left[\exp\left(\alpha \sup_{y\in B(0,1)}|W(y)|^2\right)\right].
\end{split}
\end{equation}
On the other hand,
we can take sequences $\{q_m^{(n)}\}_{m=1}^{M_n}\subset B(0,n+1)\setminus B(0,n)$
for any $n\in {\mathbb N}$ such that $M_n\asymp n^{d-1}$ and
\begin{equation}\label{eq:torus}
B(0,n+1)\setminus B(0,n)\subset \bigcup_{m=1}^{M_n}B(q_m^{(n)},1), \quad n\in {\mathbb N}.
\end{equation}
Hence, \eqref{eq:c-ineq} implies that for any positive sequence $\{a_n\}$
and $n\in {\mathbb N}$,
\[
P\left(\sup_{y\in B(q_m^{(n)},1)}|W(y)-W(q_m^{(n)})|\ge a_n\right)
\le \frac{1}{e^{\alpha a_n^2}}E\left[\exp\left(\alpha \sup_{y\in B(0,1)}|W(y)|^2\right)\right], \quad 1\le m\le M_n,
\]
and so
\begin{equation*}
\begin{split}
\sum_{n=1}^{\infty}\sum_{m=1}^{M_n}P\left(\sup_{y\in B(q_m^{(n)},1)}|W(y)-W(q_m^{(n)})|\ge a_n\right)
&\le E\left[\exp\left(\alpha \sup_{y\in B(0,1)}|W(y)|^2\right)\right]\sum_{n=1}^{\infty}\frac{M_n}{e^{\alpha a_n^2}}\\
&\le c_1\sum_{n=1}^{\infty}\frac{n^{d-1}}{e^{\alpha a_n^2}}.
\end{split}
\end{equation*}
In particular, if we let
\[
a_n=\sqrt{\frac{1}{\alpha}\left(d\log n+(1+\varepsilon)\log\log n\right)} \quad \hbox{ for some } \varepsilon>0,
\]
then
\[
\sum_{n=1}^{\infty}\frac{n^{d-1}}{e^{\alpha a_n^2}}=\sum_{n=1}^{\infty}\frac{1}{n(\log n)^{1+\varepsilon}}<\infty.
\]
Therefore, the Borel-Cantelli lemma implies that $Q$-a.s., there exists $N_0\in {\mathbb N}$ such that
\[
\sup_{y\in B(q_m^{(n)},1)}|W(y)-W(q_m^{(n)})|<\sqrt{\frac{1}{\alpha}\left(d\log n+(1+\varepsilon)\log\log n\right)},
\quad n\ge N_0, \ 1\le m\le M_n.
\]

By \cite[Theorem 9 and Lemma 1(iii)]{K75}, we see that under Assumption \ref{assum:1},
$Q$-a.s., there exists a sequence $\{x_n\}\subset {\mathbb R}^d$ such that
$|x_n|\rightarrow\infty$ as $n\rightarrow\infty$ and
\[
W(x_n)\le -\gamma(x_n)\sqrt{\log \log |x_n|}, \quad n\ge 1.
\]
Moreover, it follows by \eqref{eq:torus} that,
for any $l\ge 1$ with $|x_l|\ge N_0$,
there exist $N_l\in {\mathbb N}$ and $m_l\in \{1,\dots, M_{N_l}\}$ such that
$N_l\ge N_0$, $N_l<|x_l|\le N_l+1$ and $|x_l-q_{m_l}^{(N_l)}|\le 1$.
Hence for any $x\in B(q_{m_l}^{(N_l)},1)$,
\begin{equation*}
\begin{split}
W(x)
&=(W(x)-W(q_{m_l}^{(N_l)}))+(W(q_{m_l}^{(N_l)})-W(x_l))+W(x_l)\\
&\le 2\sup_{y\in B(q_{m_l}^{(N_l)},1)}|W(y)-W(q_{m_l}^{(N_l)})|+W(x_l)\\
&\le 2\sqrt{\frac{1}{\alpha}\left(d\log N_l +(1+\varepsilon)\log\log N_l \right)}-\gamma(x_l)\sqrt{\log\log |x_l|}\\
&\le 2\sqrt{\frac{1}{\alpha}\left(d\log |x_l| +(1+\varepsilon)\log\log |x_l| \right)}-\gamma(x_l)\sqrt{\log\log |x_l|}.
\end{split}
\end{equation*}
Then by \eqref{eq:gamma-1},
there exists $L_0\in {\mathbb N}$, such that for all $l\ge L_0$,
the last expression above is negative.
Therefore, if we write $m$ for the Lebesgue measure on ${\mathbb R}^d$,
then $Q$-a.s.,
\begin{equation*}
\begin{split}
\mu^W(B(0,R))=\int_{|x|\le R}e^{-W(x)}\,{\rm d}x
&\ge \sum_{l=L_0}^{\infty}\int_{|x|\le R, \, |x-q_{m_l}^{(N_l)}|\le 1}e^{-W(x)}\,{\rm d}x\\
&\ge \sum_{l=L_0}^{\infty} m(B(0,R)\cap B(q_{m_l}^{(N_l)},1))
\rightarrow\infty, \quad R\rightarrow\infty.
\end{split}
\end{equation*}
The proof of (1) is complete.

We turn to the proofs of (2) and (3).
By \cite[Theorem 4 and Lemma 1(iii)]{K75},
$Q$-a.s., there exists $R_0\ge 0$ such that for any $x\in {\mathbb R}^d$ with $|x|\ge R_0$,
\[
W(x)\ge -\gamma(x)\sqrt{3 \log\log |x|}.
\]
Hence we have $Q$-a.s., for all $R\ge R_0$,
\[
\int_{R_0<|x|\le R}e^{-W(x)}\,{\rm d}x \le \int_{R_0<|x|\le R}e^{\gamma(x)\sqrt{3\log\log |x|}}\,{\rm d}x.
\]

We first assume that \eqref{eq:gamma-2} holds.
Then there exist $R_1\ge R_0$ and $C_1>0$ such that for all $R\ge R_1$,
\begin{equation*}
\begin{split}
\int_{R_0<|x|\le R}e^{\gamma(x)\sqrt{3\log\log |x|}}\,{\rm d}x
&=\int_{R_0<|x|\le R_1}e^{\gamma(x)\sqrt{3\log\log |x|}}\,{\rm d}x
+\int_{R_1<|x|\le R}e^{\gamma(x)\sqrt{3\log\log |x|}}\,{\rm d}x\\
&\le \int_{R_0<|x|\le R_1}e^{\gamma(x)\sqrt{3\log\log |x|}}\,{\rm d}x
+\int_{R_1<|x|\le R}
e^{C_1|x|}\,{\rm d}x \\
&\asymp e^{C_1 R}R^{d-1}.
\end{split}
\end{equation*}
Since this implies that
\[
\limsup_{R\rightarrow\infty}\frac{1}{R}\log  \left(\mu^W(B(0,R))\right)\le C_1, \quad \text{$Q$-a.s.,}
\]
the proof of (2) is complete.

We next assume that \eqref{eq:gamma-3} holds.
Then there exist $R_2\ge R_0$ and $C_2>0$ such that for all $R\ge R_2$,
\begin{equation*}
\begin{split}
\int_{R_0<|x|\le R}e^{\gamma(x)\sqrt{3\log\log |x|}}\,{\rm d}x
&=\int_{R_0<|x|\le R_2}e^{\gamma(x)\sqrt{3\log\log |x|}}\,{\rm d}x
+\int_{R_2<|x|\le R}e^{\gamma(x)\sqrt{3\log\log |x|}}\,{\rm d}x\\
&\le \int_{R_0<|x|\le R_2}e^{\gamma(x)\sqrt{3\log\log |x|}}\,{\rm d}x
+\int_{R_2<|x|\le R}
e^{C_2|x|^p}
\,{\rm d}x\\
&\lesssim e^{C_2 R^p}R^{d-p}.
\end{split}
\end{equation*}
Since $p\in (0,1)$, we have
\[
\lim_{R\rightarrow\infty}\frac{1}{R}\log  \left(\mu^W(B(0,R))\right)=0, \quad \text{$Q$-a.s.}
\]
and so the proof of (3) is complete.
\end{proof}

\begin{proof}[Proof of Theorem {\rm \ref{thm:gauss}}]
The regularity of $({\cal E}^w,{\cal F}^w)$ implies that
the coordinate functions $\varphi_i(x)=x_i \ (1\le i \le d)$ are locally in ${\cal F}^W$
(see \cite[p.~130]{FOT11} for the notion of ``locally in ${\cal F}$'').
Then by the similar argument as in the proof of \cite[Theorem 1.4.2 (ii)]{FOT11},
we see that the function $\rho(x)=|x|=(\sum_{k=1}^dx_k^2)^{1/2}$ is also locally in ${\cal F}^W$
and $|\nabla \rho|=1$, $\mu^W$-a.e.\  on ${\mathbb R}^d$.
Hence Proposition \ref{prop:volume-gamma} and \cite[Theorem 1]{N98} (see also \cite[Theorem 3.2]{S25}) yield
$\lambda_{{\rm ess}}^w<\infty$, $Q$-a.s.\ and so
$\{P_t^W\}_{t>0}$ is
noncompact, $Q$-a.s.
In particular, if \eqref{eq:gamma-3} holds, then $\lambda_{{\rm ess}}^w=0$, $Q$-a.s.
The proof is complete.
\end{proof}

\begin{exam}\rm
Let $H\in (0,1)$ be a constant.
Assume that the covariance structure of the Gaussian process $(\{W(x)\}_{x\in \R^d}, Q)$ is given by
\[
E[W(x)W(y)]=\frac{1}{2}\left(|x|^{2H}+|y|^{2H}-|x-y|^{2H}\right), \quad x,y\in \R^d,
\]
that is, $(\{W(x)\}_{x\in \R^d}, Q)$ is the fractional Brownian field with the Hurst parameter $H$.
If $H=1/2$, then $(\{W(x)\}_{x\in \R^d}, Q)$ is the so-called L\'evy's Brownian motion with multidimensional parameter.
For any $H\in (0,1)$,
since Assumption \ref{assum:1} is fulfilled with $l=1$ and $\gamma(x)=\sigma_1(|x|)=|x|^H $ for all $x\in {\mathbb R}^d$,
Theorem {\rm \ref{thm:gauss}} implies that
the Markovian semigroup $\{P_t^W\}_{t>0}$ is noncompact
with $\lambda_{{\rm ess}}^W=0$.
We note that for any $d\ge 1$ and $H\in (0,1)$,
$\{P_t^W\}_{t>0}$ is recurrent (see \cite[Theorem 1]{Tanaka93} and \cite[Theorem 4.1]{KTT15}).
\end{exam}

\begin{rem}\rm
Let $(\{B(t)\}_{t\ge 0},Q)$ be the fractional Brownian motion with the Hurst parameter $H\in (0,1)$.
Then $(\{B(t)\}_{t\ge 0},Q)$ is a Gaussian process with the covariance
\[
E[B(x)B(y)]=\frac{1}{2}\left(|x|^{2H}+|y|^{2H}-|x-y|^{2H}\right), \quad x,y\ge 0.
\]
If we take $W(x)=B(|x|)$ for $x\in \R^d$,
then Theorem \ref{thm:gauss} is not directly applicable to the semigroup $\{P_t^W\}_{t>0}$;
however, the argument in the proof of Proposition \ref{prop:volume-gamma} still works
because $\{B(t)\}_{t \ge 0}$ satisfies the law of the iterated logarithm (\cite[Theorem 1.1]{O72}).
In particular, we have $\lambda_{{\rm ess}}^W=0$, $Q$-a.s.
We note that if $H=1/2$, then $\{B(t)\}_{t\ge 0}$ is the  Brownian motion on $\R$ and
$\{P_t^W\}_{t>0}$ is recurrent for any $d\ge 1$ (\cite[Theorem 2.1]{FNT87}).
\end{rem}

\section{Semi-selfsimilar L\'evy environment}\label{e:section2.2}

In this section, we discuss the noncompactness of Markovian semigroups
generated by Brox diffusion processes in semi-selfsimilar environment.

\subsection{Formulation}\label{subsect:formulation}
Following \cite{KK20,KTT17}, we first define a random function on ${\mathbb R}^d$
associated with a semi-selfsimilar L\'evy process on ${\mathbb R}$.
Let ${\cal W}$ be the set of
functions $w: \R \to \R$ such that
$w(0)=0$, $w$ is right continuous with left limits on $[0,\infty)$, and left continuous with right limits on $(-\infty,0]$.
Let ${\cal P}$ be the set of Borel cylindrical sets in ${\cal W}$:
\[
{\cal P}=\left\{\left\{w\in {\cal W} : (w(x_1),\dots, w(x_n))\in B\right\}:
(x_1,\dots,x_n)\in {\mathbb R}^n, B\in {\cal B}({\mathbb R}^n), n\in {\mathbb N}\right\},
\]
which forms a $\pi$-system.
The $\sigma$-field generated by ${\cal P}$, say ${\cal B}({\cal W})$, coincides with the Borel $\sigma$-field generated by the uniform convergence topology on compact sets.

For $x\ge 0$, let $V_+(x)$ be a random variable on $({\cal W},{\cal B}({\cal W}))$ defined by
\[
V_+(x)(w)=w(x), \quad w\in {\cal W}.
\]
Suppose that $Q$ is a probability measure on $({\cal W},{\cal B}({\cal W}))$ so that
$(\{V_+(x)\}_{x\ge 0},Q)$ is a real-valued L\'evy process on ${\mathbb R}$ satisfying
\begin{enumerate}
\item $Q(V_+(0)=0)=1$ and $Q(V_+(x)=0)<1$ for any $x>0$;
\item $\{V_+(x)\}_{x\ge 0}$ is right continuous with left limits on $[0,\infty)$, $Q$-a.s.;
\item There exist constants $c>0$ with $c\ne 1$ and $\alpha\in (0,2)$ such that
$\{V_+(cx)\}_{x\ge 0}\overset{d}{=}\{c^{1/\alpha}V_+(x)\}_{x\ge 0}$ under the law $Q$.
\end{enumerate}
Here, for two  stochastic processes  $\{Y_1(t)\}_{t\in T}$ and $\{Y_2(t)\}_{t\in T}$
with index set $T$,
$\{Y_1(t)\}_{t\in T}\overset{d}{=}\{Y_2(t)\}_{t\in T}$ means that all of their finite dimensional distributions coincide.
In other words, $(\{V_+(x)\}_{x\ge 0},Q)$ is a nontrivial semi-selfsimilar L\'evy process on ${\mathbb R}$
with index $\alpha\in (0,2]$, and starting from $x=0$
(see \cite{MS99} and \cite[Section 13]{Sa13} for details).
We define
\[
r_0=\inf\left\{c>1 : \text{$\{V_+(cx)\}_{x\ge 0}\overset{d}{=}\{c^{1/\alpha}V_+(x)\}_{x\ge 0}$ under the law $Q$}\right\}.
\]
Then by \cite[Theorem 13.11]{Sa13}, we know that
\begin{itemize}
\item If $r_0=1$, then for any $r>0$, $\{V_+(rx)\}_{x\ge 0}\overset{d}{=}\{r^{1/\alpha}V_+(x)\}_{x\ge 0}$ under the law $Q$.
\item If $r_0>1$, then for any $n\in {\mathbb Z}$, $\{V_+(r_0^n x)\}_{x\ge 0}\overset{d}{=}\{r_0^{n/\alpha}V_+(x)\}_{x\ge 0}$ under the law $Q$.
\end{itemize}
When $r_0=1$, we call $(\{V_+(x)\}_{x\ge 0},Q)$ an $\alpha$-stable L\'evy process;
otherwise, we call $(\{V_+(x)\}_{x\ge 0},Q)$ an $(r_0,\alpha)$-semi-selfsimilar L\'evy process.

For $x\ge 0$, let $V_-(x)$ be a random variable on $({\cal W},{\cal B}({\cal W}))$ defined by
\[
V_-(x)(w)=w(-x), \quad w\in {\cal W}.
\]
Suppose also that $(\{V_-(x)\}_{x\ge 0},Q)$ is an $(r_0,\alpha)$-semi-selfsimilar L\'evy process
independently of $\{V_+(x)\}_{x\ge 0}$.
For $x\in \R$, we further define
a
random variable $V(x)$ on $({\cal W},{\cal B}({\cal W}))$ by
\[
V(x)(w)=\begin{dcases}
V_+(x)(w), & x\ge 0, \\
V_-(-x)(w), & x<0.
\end{dcases}
\]
We then have
\begin{enumerate}
\item[(i)']$Q(V(0)=0)=1$ and $Q(V(x)=0)<1$ for any $x\in {\mathbb R}\setminus \{0\}$;
\item[(ii)'] $V$ is right continuous with left limits on $[0,\infty)$, and left continuous with right limits on $(-\infty,0]$, $Q$-a.s.;
\item[(iii)'] If $r_0=1$, then for any $r>0$, $\{V(rx)\}_{x\in {\mathbb R}}\overset{d}{=}\{r^{1/\alpha}V(x)\}_{x\in {\mathbb R}}$ under the law $Q$.
On the other hand, if $r_0>1$,
then for any $n\in {\mathbb Z}$, $\{V(r_0^n x)\}_{x\in {\mathbb R}}\overset{d}{=}\{r_0^{n/\alpha}V(x)\}_{x\in {\mathbb R}}$ under the law $Q$.
\end{enumerate}
For $x=(x_1,\dots, x_d)\in {\mathbb R}^d$, we define
a
random variable $W(x)$ by
\[
W(x)=\sum_{k=1}^d V(x_k).
\]
Then $\{W(x)\}_{x\in {\mathbb R}^d}$ forms a random function in ${\mathbb R}^d$
associated with $\{V(x)\}_{x\in {\mathbb R}}$.

As in the Gaussian environment model,
for each $Q$-a.s.\ realization of $\{W(x)\}_{x\in {\mathbb R}^d}$,
let $\mu^W({\rm d}x)=e^{-W(x)}\,{\rm d}x$ be a positive Radon measure on ${\mathbb R}^d$ with full support.
Let $\{P_t^W\}_{t>0}$ be the Markovian semigroup on $L^2(\R^d;\mu^W)$
associated with $({\cal E}^W,{\cal F}^W)$,
and let $\lambda_{{\rm ess}}^W$ be the bottom of the essential spectrum for $-L^W$.
We then have

\begin{thm}\label{thm:semi}
Suppose that $\alpha\in (1,2)$ and $Q(V(1)<0)>0$.
Then $\lambda_{{\rm ess}}^W=0$, $Q$-a.s.
In particular, $\{P_t^W\}_{t>0}$ is
noncompact, $Q$-a.s.
\end{thm}

Theorem \ref{thm:semi} shows that
the random drift perturbation does not change the value of the bottom of the essential spectrum.
We note that Theorem \ref{thm:semi} is applicable to the operators $-L^W$ with any $d\ge 1$.
On the other hand, we do not know if Theorem \ref{thm:semi} is valid or invalid for $\alpha\in (0,1]$.
We will show later that if $d=1$ and $\{V_+(x)\}_{x \ge 0}$ is a symmetric $\alpha$-stable process with any $\alpha\in (0,2)$,
then the associated semigroup is noncompact (see Proposition \ref{prop:symm-stable} below).

\subsection{Proof of Theorem \ref{thm:semi}}
In the same way as the proof of Theorem \ref{thm:gauss},
we establish Theorem \ref{thm:semi} by investigating the almost sure volume growth behavior of $\mu^W$.
In order to do so, we reveal sample path properties of semi-selfsimilar processes.

We first discuss the ergodic property of the scaling transform.
For $r>0$, we define the map $T_r:{\cal W}\to {\cal W}$ by
\[
T_rw(x)=r^{-1/\alpha}w(rx), \quad x\in {\mathbb R}, \ w\in {\cal W}.
\]
If $r_0=1$, then for any $r>0$,
$T_r$ is a measure preserving transform on the probability space $({\cal W},{\cal B}({\cal W}),Q)$.
On the other hand, if $r_0>1$, then $T_{r_0}$ satisfies the same property.
For $n\ge 2$, we define $T_{r}^n$ as the $n$-th iteration of $T_r$.
The next lemma is mentioned in \cite[p.~244]{FR05},
but we present its proof for the readers' convenience.
\begin{lem}\label{lem:scale-ergodic}
Let $r>1$ when $r_0=1$, and let
$r=r_0$ when $r_0>1$.
Then for any $A,B\in {\cal B}({\cal W})$,
\begin{equation}\label{eq:mixing}
\lim_{n\rightarrow\infty}Q((T_{r}^{-n}A)\cap B)=Q(A)Q(B).
\end{equation}
Namely, $\{T_{r}^n\}_{n\in {\mathbb N}}$ is strongly mixing and so ergodic.
\end{lem}

\begin{proof}
According to \cite[p.~41, Theorem 1.17 (iii)]{W82}, it suffices to show \eqref{eq:mixing} for any $A,B\in {\cal P}$.
Following \cite{I44}, we prove this by way of the characteristic functions.

We write $T=T_r$ for simplicity.
Recall that, according to Subsection \ref{subsect:formulation} (iii)', for all $n\ge1$, $\{T^nV(x)\}_{x\in \R}$ and $\{V(x)\}_{x\in \R}$ have the same distribution.
Let $\{s_j\}_{j=1}^l$ and $\{t_k\}_{k=1}^m$ be any positive increasing sequences.
Take $n\ge 1$ so large that $r^ns_1>t_m$.
Then, for any real sequences $\{\xi_j\}_{j=1}^l$ and $\{\eta_k\}_{k=1}^m$,
we have
\begin{equation*}
\begin{split}
&\prod_{j=1}^l \exp\left(i\xi_j (T^nV(s_j))\right)\prod_{k=1}^m \exp\left(i \eta_k  V(t_k)\right)\\
&=\prod_{j=1}^l \exp\left(i\xi_j (r^{-n/\alpha}V(r^ns_j))\right)\prod_{k=1}^m \exp\left(i \eta_k  V(t_k)\right)\\
&=\prod_{j=2}^l \exp\left(i\left(\sum_{p=j}^l \xi_p\right) r^{-n/\alpha}(V(r^ns_j)-V(r^ns_{j-1}))\right)\\
&\quad \times \exp\left(i\left(\sum_{p=1}^l \xi_p\right) r^{-n/\alpha}(V(r^ns_1)-V(t_m))\right)\\
&\quad\times \prod_{k=1}^{m-1} \exp\left(i \eta_k  V(t_k)\right)
\times \exp\left(i \left(r^{-n/\alpha}\sum_{p=1}^l \xi_p +\eta_m\right)  V(t_m)\right).
\end{split}
\end{equation*}
Taking the expectations
above, we obtain
\begin{equation}\label{eq:ergodic-1}
\begin{split}
&E\left[\prod_{j=1}^l \exp\left(i\xi_j (T^nV(s_j))\right)\prod_{k=1}^m \exp\left(i \eta_k  V(t_k)\right)\right]\\
&=E\Biggl[\prod_{j=2}^l \exp\left(i\left(\sum_{p=j}^l \xi_p\right) r^{-n/\alpha}(V(r^n s_j)-V(r^n s_{j-1}))\right)\\
&\quad\quad\quad\times \exp\left(i\left(\sum_{p=1}^l \xi_p\right) r^{-n/\alpha}(V(r^n s_1)-V(t_m))\right)\Biggr]\\
&\quad\times E\left[\prod_{k=1}^{m-1} \exp\left(i \eta_k  V(t_k)\right) \exp\left(i \left(r^{-n/\alpha}\sum_{p=1}^l \xi_p +\eta_m\right)  V(t_m)\right)\right]\\
&={\rm (I)}_n\times {\rm (II)}_n,
\end{split}
\end{equation}
where the first equality above follows by the independent increments property of $\{V(x)\}_{x\in {\mathbb R}}$.

By the dominated convergence theorem, we have
\begin{equation}\label{eq:ergodic-2}
\lim_{n\rightarrow\infty}{\rm (II)}_n=E\left[\prod_{k=1}^m \exp\left(i \eta_k  V(t_k)\right)\right].
\end{equation}
On the other hand, by the stationary increment property and the scaling property of $\{V(x)\}_{x\in {\mathbb R}}$,
we also have
\begin{equation*}
\begin{split}
{\rm (I)}_n
&=\prod_{j=2}^l E\left[\exp\left(i\left(\sum_{p=j}^l \xi_p\right) r^{-n/\alpha}V(r^n(s_j-s_{j-1}))\right)\right]\\
&\quad\times E\left[\exp\left(i\left(\sum_{p=1}^l \xi_p\right) r^{-n/\alpha}V(r^n(s_1-r^{-n}t_m))\right)\right]\\
&=\prod_{j=2}^l E\left[\exp\left(i\left(\sum_{p=j}^l \xi_p\right) V(s_j-s_{j-1})\right)\right]
E\left[\exp\left(i\left(\sum_{p=1}^l \xi_p\right) V(s_1-r^{-n}t_m)\right)\right].
\end{split}
\end{equation*}
Since $\{V(x)\}_{x\in {\mathbb R}}$ is stochastically continuous,
we obtain $V(s_1-r^{-n}t_m)\rightarrow V(s_1) \ (n\rightarrow\infty)$ in probability,
and so
\[
\lim_{n\rightarrow\infty}E\left[\exp\left(i\left(\sum_{p=1}^l \xi_p\right) V(s_1-r^{-n}t_m)\right)\right]
=E\left[\exp\left(i\left(\sum_{p=1}^l \xi_p\right) V(s_1)\right)\right].
\]
We thus get
\begin{equation}\label{eq:ergodic-3}
\begin{split}
\lim_{n\rightarrow\infty}{\rm (I)}_n
&=\prod_{j=2}^l E\left[\exp\left(i\left(\sum_{p=j}^l \xi_p\right) V(s_j-s_{j-1})\right)\right]E\left[\exp\left(i\left(\sum_{p=1}^l \xi_p\right) V(s_1)\right)\right]\\
&=E\left[\prod_{j=1}^l \exp\left(i\xi_j V(s_j)\right)\right],
\end{split}
\end{equation}
where the second equality above follows by  the independent and stationary increments property of $\{V(x)\}_{x\in {\mathbb R}}$ again.

By \eqref{eq:ergodic-1}, \eqref{eq:ergodic-2} and \eqref{eq:ergodic-3}, we have
\begin{equation}\label{eq:ch-mixing}
\begin{split}
&\lim_{n\rightarrow\infty}
E\left[\prod_{j=1}^l \exp\left(i\xi_j (r^{-n/\alpha}V(r^ns_j))\right)\prod_{k=1}^m \exp\left(i \eta_k  V(t_k)\right)\right]\\
&=E\left[\prod_{j=1}^l \exp\left(i\xi_j V(s_j)\right)\right]E\left[\prod_{k=1}^m \exp\left(i \eta_k  V(t_k)\right)\right].
\end{split}
\end{equation}
Note that this equality is valid
even if $\{s_j\}_{j=1}^l$ and $\{t_k\}_{k=1}^m$ are real-valued increasing sequences.
In fact, we can follow the argument for the proof of \eqref{eq:ch-mixing}
by separating the positive and negative parts of $\{s_j\}_{j=1}^l$ and $\{t_k\}_{k=1}^m$, respectively.
We thus get \eqref{eq:mixing}.
\end{proof}

Let $S^{d-1}$ be the surface of the unit ball in $\R^d$ centered at the origin.
The next lemma can be proved by the argument of \cite[Proposition 2.1]{KTT17}
applied to the process $\{-V(x)\}_{x\in {\mathbb R}}$.

\begin{lem}\label{lem:positive}
Suppose that $Q(V(1)<0)>0$.
Then for any $c>0$,
\[
Q\left(\sup_{u\in [1,r_0], \, \sigma\in S^{d-1}}\left(\sum_{k=1}^d V (u\sigma_k)\right)\le -c\right)>0.
\]
\end{lem}

Following \cite{K38}, we next prove the almost sure upper bound of sample paths for $\{V(x)\}_{x\in \R}$.
\begin{lem}\label{lem:upper}
Let $q>1/\alpha$ and $\kappa>0$.
Then $Q$-a.s., there exists $R_0>0$ such that for all $x\in {\mathbb R}$ with $|x|\ge R_0$,
$|V(x)|\le \kappa|x|^{1/\alpha}(\log |x|)^q$.
\end{lem}

\begin{proof}
By the definition of $\{V(x)\}_{x\in {\mathbb R}}$,
it is enough to show the assertion for $x\ge 0$ only.
If $r_0=1$, then $\{V(x)\}_{x\ge 0}$ is a symmetric $\alpha$-stable process,
and so the desired assertion follows by \cite{K38}.
Therefore, we may and do assume that $r_0>1$ in the subsequent argument.

Since $V(1)$ follows a strictly $\alpha$-semistable distribution,
we have $E[|V(1)|^p]<\infty$ for any $p\in [0,\alpha)$ (see \cite[Remark 1]{C94} or \cite[Example 25.10]{Sa13}).
As $\{V(x)\}_{x\ge 0}$ is a L\'evy process
and the function $g(x)=\max\{|x|^p,1\}$ is submultiplicative
(see \cite[Scholium 5.18 in p.~444 and Example in p.~445]{KS23} or \cite[Definition 25.2]{Sa13}),
we further have for any $p\in [0,\alpha)$, $x\ge0$ and $t\ge0$,
$E[|V(x)|^p]<\infty$ and $E[\sup_{y \in [0,t]}|V(y)|^p]<\infty$
 (see \cite[Theorem 5.19 in p.~445]{KS23} or \cite[Theorem  25.18]{Sa13}).

For $q>0$ and $u\ge 1$, let $c(u)=u^{1/\alpha}(\log u)^q$.
Then for any $p\in [0,\alpha)$ and $n\ge 1$,
we have by Chebyshev's inequality,
\[
Q(|V(r_0^n)|\ge c(r_0^n))\le \frac{1}{c(r_0^n)^p}E[|V(r_0^n)|^p]
=\frac{r_0^{np/\alpha}}{c(r_0^n)^p}E[|V(1)|^p]
=\frac{1}{n^{pq}(\log r_0)^{pq}}E[|V(1)|^p].
\]
In particular, for any $q>1/\alpha$, we can take some constant $p\in [0,\alpha)$ with $pq>1$ so that
\[
\sum_{n=1}^{\infty}Q(|V(r_0^n)|\ge c(r_0^n))
\le E[|V(1)|^p] \sum_{n=1}^{\infty}\frac{1}{n^{pq}(\log r_0)^{pq}}<\infty.
\]
Hence by the Borel-Cantelli lemma, $Q$-a.s.,
there exists $N_0\in {\mathbb N}$ such that for all $n\ge N_0$,
$|V(r_0^n)|\le c(r_0^n)$.

For $n\ge 1$, since $\{V(x)\}_{x\ge 0}$ has right continuous sample paths and
\[
\{V(t)-V(r_0^n)\}_{r_0^n\le t\le r_0^{n+1}}
\overset{d}{=}\{V(t-r_0^n)\}_{r_0^n\le t\le r_0^{n+1}}
\overset{d}{=}\{r_0^{n/\alpha}V(t-1)\}_{1\le t\le r_0},
\]
we have
\begin{equation*}
\begin{split}
&Q\left(\max_{t\in [r_0^n, r_0^{n+1}]\cap{\mathbb Q}}|V(t)-V(r_0^n)|\ge c(r_0^n)\right)\\
&=Q\left(r_0^{n/\alpha}\max_{t\in [1,r_0]\cap (r_0^{-n}{\mathbb Q})}|V(t-1)|\ge c(r_0^n)\right)
=Q\left(\max_{t\in [1,r_0]\cap (r_0^{-n}{\mathbb Q})}|V(t-1)|\ge r_0^{-n/\alpha}c(r_0^n)\right)\\
&\le \frac{1}{(r_0^{-n/\alpha}c(r_0^n))^p}E\left[\max_{t\in [1,r_0]\cap (r_0^{-n}{\mathbb Q})}|V(t-1)|^p \right]
\le \frac{1}{n^{pq}(\log r_0)^{pq}}E\left[\max_{t\in [1,r_0]}|V(t-1)|^p \right]
\end{split}
\end{equation*}
and so
\[
\sum_{n=1}^{\infty}Q\left(\max_{t\in [r_0^n, r_0^{n+1}]\cap{\mathbb Q}}|V(t)-V(r_0^n)|\ge c(r_0^n)\right)
\le E\left[\max_{t\in [1,r_0]}|V(t-1)|^p \right]\sum_{n=1}^{\infty}\frac{1}{n^{pq}(\log r_0)^{pq}}<\infty.
\]
Then by the Borel-Cantelli lemma again, $Q$-a.s.,
there exists $N_1\in {\mathbb N}$ such that for all $n\ge N_1$,
$\max_{t\in [r_0^n, r_0^{n+1}]\cap{\mathbb Q}}|V(t)-V(r_0^n)|\le c(r_0^n)$.

Let $N=N_0\vee N_1$, and take any $x\in [r_0^N,\infty)\cap {\mathbb Q}$.
Then there exists $n\ge N$ such that $r_0^{n}\le x\le r_0^{n+1}$ and so
\[
|V(x)|\le |V(r_0^n)|+\max_{t\in [r_0^n, r_0^{n+1}]\cap{\mathbb Q}}|V(t)-V(r_0^n)|\le 2c(r_0^n)\le 2c(x)=2x^{1/\alpha}(\log x)^q.
\]
The last inequality above follows by the monotone increasing property of the function $c(u)$.
By the right continuity of sample paths of $\{V(x)\}_{x\ge 0}$,
we also have $|V(x)|\le 2x^{1/\alpha}(\log x)^q$ for all $x\ge r_0^N$.
Moreover, if we replace the function $c(u)$ by $\kappa u^{1/\alpha}(\log u)^q/2$,
then the same argument above yields $|V(x)|\le \kappa x^{1/\alpha}(\log x)^q$ for all $x\ge r_0^N$.
The proof is complete.
\end{proof}

Recall that $\mu^W({\rm d}x)=e^{-W(x)}\,{\rm d}x$.
Then by Lemmas \ref{lem:scale-ergodic}, \ref{lem:positive} and \ref{lem:upper},
we obtain
\begin{prop}\label{prop:volume--stable}
Suppose that $Q(V(1)<0)>0$. Then the next assertions hold.
\begin{enumerate}
\item[{\rm (1)}] $\lim_{R\rightarrow\infty}\mu^W(B(0,R))=\infty$, $Q$-a.s.
\item[{\rm (2)}] Suppose that $\alpha\in (1,2)$. Then
\begin{equation}\label{eq:volume-stable}
\lim_{R\rightarrow\infty}\frac{1}{R}\log \mu^W(B(0,R))=0, \quad \text{$Q$-a.s.}
\end{equation}
\end{enumerate}
\end{prop}

\begin{proof}
We first prove (1) by following the proof of \cite[Theorem 1.1]{KTT17}.
If $r_0=1$, then we let $r$ be any constant satisfying $r>1$.
If $r_0>1$, then we take  $r=r_0$.
By the definition of $W(x)$,
\begin{equation*}
\begin{split}
\mu^W(B(0,R))
=\int_{|x|\le R}e^{-W(x)}\,{\rm d}x
&=\int_{|x|\le R}e^{-\sum_{k=1}^dV(x_k)}\,{\rm d}x_1 \cdots {\rm d}x_d\\
&=\int_0^R \left(\int_{S^{d-1}}e^{-\sum_{k=1}^d V(s\sigma_k)}\,{\rm d}\sigma\right) s^{d-1}\,{\rm d}s
\end{split}
\end{equation*}
with $\sigma=(\sigma_1,\dots, \sigma_d)\in S^{d-1}$.
For $R\ge 1$, let $n$ be a positive integer such that $r^n\le R< r^{n+1}$.
Then
\begin{equation*}
\begin{split}
\int_0^R \left(\int_{S^{d-1}}e^{-\sum_{k=1}^d V(s\sigma_k)}\,{\rm d}\sigma\right) s^{d-1}\,{\rm d}s
&\ge \int_1^{r^n} \left(\int_{S^{d-1}}e^{-\sum_{k=1}^d V(s\sigma_k)}\,{\rm d}\sigma\right) s^{d-1}\,{\rm d}s \\
&=\sum_{m=0}^{n-1} \int_{r^m}^{r^{m+1}} \left(\int_{S^{d-1}}e^{-\sum_{k=1}^d V(s\sigma_k)}\,{\rm d}\sigma\right) s^{d-1}\,{\rm d}s\\
&= \sum_{m=0}^{n-1} r^{dm} \int_1^r \left(\int_{S^{d-1}}e^{-\sum_{k=1}^d V(r^m u\sigma_k)}\,{\rm d}\sigma\right) u^{d-1}\,{\rm d}u.
\end{split}
\end{equation*}
At the last equality above, we used the change of variables formula with $s=r^m u$.

Let
\[
M(m)=\sup_{u\in [1,r], \, \sigma\in S^{d-1}}\left(\sum_{k=1}^d r^{-m/\alpha}V(r^m u\sigma_k)\right).
\]
Then by (iii)' in Subsection \ref{subsect:formulation}, we have $M(m)\overset{d}{=}M(0)$ under the law $Q$.
Let $\omega_d$ be the surface area of the unit ball in $\R^d$. Then
\begin{equation*}
\begin{split}
&\sum_{m=0}^{n-1} r^{dm} \int_1^r \left(\int_{S^{d-1}}e^{-\sum_{k=1}^d V(r^m u\sigma_k)}\,{\rm d}\sigma\right) u^{d-1}\,{\rm d}s\\
&\ge \sum_{m=0}^{n-1} r^{dm} \int_1^r \left(\int_{S^{d-1}}e^{-r^{m/\alpha}M(m)}\,{\rm d}\sigma\right) u^{d-1}\,{\rm d}s\\
&=\frac{\omega_d}{d}(r^d-1)\sum_{m=0}^{n-1} r^{dm}e^{-r^{m/\alpha}M(m)}
=\frac{\omega_d}{d}(r^d-1)\sum_{m=0}^{n-1} e^{dm\log r-r^{m/\alpha}M(m)}.
\end{split}
\end{equation*}
Moreover, for any $c>0$,
\begin{equation*}
\begin{split}
\frac{\omega_d}{d}(r^d-1)\sum_{m=0}^{n-1} e^{dm\log r-r^{m/\alpha}M(m)}
&\ge \frac{\omega_d}{d}(r^d-1)\sum_{m=0}^{n-1} e^{dm\log r-r^{m/\alpha}M(m)}{\bf 1}_{(-\infty,-c]}(M(m))\\
&\ge \frac{\omega_d}{d}(r^d-1)\sum_{m=0}^{n-1} e^{dm\log r+cr^{m/\alpha}}{\bf 1}_{(-\infty,-c]}(M(m))\\
&\ge \frac{\omega_d}{d}(r^d-1)\sum_{m=0}^{n-1} {\bf 1}_{(-\infty,-c]}(M(m)).
\end{split}
\end{equation*}
We thus have
\begin{equation}\label{eq:volume-lower}
\mu^W(B(0,R))\ge \frac{\omega_d}{d}(r^d-1)\sum_{m=0}^{n-1} {\bf 1}_{(-\infty,-c]}(M(m)).
\end{equation}

Since
\[
M(0)=\sup_{u\in [1,r], \, \sigma\in S^{d-1}}\left(\sum_{k=1}^d V(u\sigma_k)\right),
\]
Lemma \ref{lem:positive} yields $Q(M(0)\le -c)>0$.
Moreover,
by the Birkhoff ergodic theorem (see, e.g., \cite[p.~34, Theorem 1.14]{W82}) and Lemma \ref{lem:scale-ergodic},
we obtain $Q$-a.s.,
\[
\lim_{n\rightarrow\infty}\frac{1}{n}\sum_{m=0}^{n-1} {\bf 1}_{(-\infty,-c]}(M(m))
=E[{\bf 1}_{(-\infty,-c]}(M(0))]=Q(M(0)\le -c)>0
\]
and so
\begin{equation}\label{eq:ergodic}
\lim_{n\rightarrow\infty}\sum_{m=0}^{n-1} {\bf 1}_{(-\infty,-c]}(M(m))=\infty.
\end{equation}
Hence by \eqref{eq:volume-lower},
\[
\liminf_{R\rightarrow\infty}\mu^W(B(0,R))
\ge \frac{\omega_d}{d}(r^d-1)\lim_{n\rightarrow\infty}\sum_{m=0}^{n-1} {\bf 1}_{(-\infty,-c]}(M(m))=\infty,
\quad \text{$Q$-a.s.}
\]
This completes the proof of (1).

We next prove (2).
Since $B(0,R)\subset [-R,R]^d$, we have
\begin{equation}\label{eq:volume-upper-1}
\begin{split}
\mu^W(B(0,R))
&=\int_{|x|\le R}e^{-W(x)}\,{\rm d}x
=\int_{|x|\le R}e^{-\sum_{k=1}^d V(x_k)}\,{\rm d}x_1\cdots{\rm d}x_d\\
&\le \int_{[-R,R]^d}e^{-\sum_{k=1}^d V(x_k)}\,{\rm d}x_1\cdots{\rm d}x_d
=\left(\int_{-R}^Re^{-V(t)}\,{\rm d}t\right)^d.
\end{split}
\end{equation}
Fix a constant $p>1/\alpha$.
Then by Lemma \ref{lem:upper},  $Q$-a.s.,
there exists $R_0>0$ such that for any $x\in {\mathbb R}$ with $|x|\ge R_0$,
\[
V(x)\ge -|x|^{1/\alpha}(\log |x|)^p.
\]
Hence for any $R>R_0$,
\begin{equation}\label{eq:volume-upper-2}
\begin{split}
\int_{-R}^Re^{-V(t)}\,{\rm d}t
&=\int_{|t|\le R_0}e^{-V(t)}\,{\rm d}t+\int_{R_0<|t|\le R}e^{-V(t)}\,{\rm d}t\\
&\le \int_{|t|\le R_0}e^{-V(t)}\,{\rm d}t+\int_{R_0<|t|\le R}e^{|t|^{1/\alpha}(\log|t|)^p}\,{\rm d}t.
\end{split}
\end{equation}
Noting that
\[
\int_{R_0<|t|\le R}e^{|t|^{1/\alpha}(\log|t|)^p}\,{\rm d}t
\sim 2\alpha \frac{\exp\left(R^{1/\alpha}(\log R)^p\right)R^{1-1/\alpha}}{(\log R)^p}, \quad R\rightarrow\infty,
\]
we obtain by \eqref{eq:volume-upper-1} and \eqref{eq:volume-upper-2},
\begin{equation*}
\begin{split}
\mu^W(B(0,R))
\le \left(\int_{-R}^Re^{-V(t)}\,{\rm d}t\right)^d
&\le \left(\int_{|t|\le R_0}e^{-V(t)}\,{\rm d}t+\int_{R_0<|t|\le R}e^{|t|^{1/\alpha}(\log|t|)^p}\,{\rm d}t\right)^d\\
&\sim (2\alpha)^d \frac{\exp\left(dR^{1/\alpha}(\log R)^p\right)R^{d(1-1/\alpha)}}{(\log R)^{dp}}, \quad R\rightarrow\infty.
\end{split}
\end{equation*}
Hence for $\alpha\in (1,2)$, we get \eqref{eq:volume-stable}.
\end{proof}

\begin{proof}[Proof of Theorem {\rm \ref{thm:semi}}]
By following the proof of Theorem \ref{thm:gauss} together with Proposition \ref{prop:volume--stable},
we have the desired assertion.
\end{proof}

\subsection{Symmetric stable environment in $d=1$}

In Theorem \ref{thm:semi}, we assumed that $\alpha\in (1,2)$.
We now drop this restriction under the additional condition that
$(\{V_+(x)\}_{x\ge 0}, Q)$ is a one-dimensional symmetric stable process.
Here the symmetry means that the two processes
$\{V_+(x)\}_{x\ge 0}$ and $\{-V_+(x)\}_{x\ge 0}$ have the same finite dimensional distribution under the law $Q$.
Note that if $d=1$, then $W(x)=V(x)$ for $x\in {\mathbb R}$.

\begin{prop}\label{prop:symm-stable}
Suppose that $d=1$ and $(\{V_+(x)\}_{x\ge 0}, Q)$ is a symmetric stable process with index $\alpha\in (0,2)$.
Then $\lambda_{{\rm ess}}^W=0$, $Q$-a.s.
In particular, $\{P_t^W\}_{t>0}$ is noncompact in $L^2({\mathbb R};\mu^W)$, $Q$-a.s.
\end{prop}

\begin{proof}
For $d=1$, the generator $L^W$ is written as
\[
L^W=\frac{1}{2}\frac{{\rm d}}{{\rm d}\mu^W}\frac{{\rm d}}{{\rm d}s^W}, \quad s^W(x)=\int_0^x e^{W(y)}\,{\rm d}y.
\]
If $\alpha\in [1,2)$, then $\{V_+(x)\}_{x\ge 0}$ is recurrent.
On the other hand, if $\alpha\in (0,1)$, then
$\{V_+(x)\}_{x\ge 0}$ is transient with accumulation points $\pm\infty$ (see, e.g., \cite[p.~365, Theorem 48.6]{Sa13}).
Hence for any $\alpha\in (0,2)$, we have by \cite[Theorem 1]{BY05},
\[
Q\left(\int_0^{\infty}e^{V_+(x)}\,{\rm d}x=\infty, \ \int_0^{\infty}e^{-V_+(x)}\,{\rm d}x=\infty\right)=1.
\]
This implies
\[
Q\left(\int_{-\infty}^0e^{V_-(x)}\,{\rm d}x=\infty, \ \int_{-\infty}^0 e^{-V_-(x)}\,{\rm d}x=\infty\right)=1
\]
and so
\[
Q\left(\int_{-\infty}^0 e^{\pm W(x)}\,{\rm d}x=\infty, \ \int_0^{\infty} e^{\pm W(x)}\,{\rm d}x=\infty\right)=1.
\]
Then by \cite[Theorem 2]{P09} applied to $L^W$, we arrive at the desired assertion.
\end{proof}

\section{Perturbation of diffusion operators by random drift}\label{section:per}
As we discussed in the previous sections,
the random drift perturbation does not change the noncompactness property
for the Laplace operator.
In this section, we examine if the random drift perturbation can change the compactness property
for some diffusion operators.
For $\delta>0$, let $w_1(x)=|x|^{\delta}$  and $w_2(x)=-|x|^{\delta}$ for $x\in \R^d$.
As mentioned in Section \ref{sect:prelim}, for each $k=1,2$, the diffusion operator $L^{w_k}=\frac{1}{2}\Delta-\nabla w_k \cdot \nabla$ on $\R^d$
generates a compact Markovian semigroup in $L^2(\R^d;\mu^{w_k})$ if and only if $\delta>1$.
Taking this fact into account, we focus on the noncompactness property
for the random drift perturbation of $L^{w_k}$ under the self-similar environment setting.

In what follows, we assume that $d=1$ because we will use the result of Pinsky \cite{P09}
on the (non)compactness criterion for one-dimensional diffusion operators.
We also keep the same setting as in Section \ref{e:section2.2}.
Since we assume that $d=1$, we have $W(x)=V(x)$.
For $n\ge 1$, let $m(n)=\inf_{u\in [1,r_0]}T_{r_0}^n V(u)$.
Then $m(n)$ and $m(1)$ are identically distributed by (iii)' in Subsection \ref{subsect:formulation}.

\begin{lem}\label{lem:divergent}
Suppose that $Q(V(1)>0)>0$ and $Q(V(1)<0)>0$, that is,
$\{V(x)\}_{x\in {\mathbb R}}$ changes the sign,  $Q$-a.s.
\begin{enumerate}
\item[{\rm (i)}]
For any $c>0$, $Q(m(1)>c)>0$.

\item[{\rm (ii)}] If $0<\delta\le 1/\alpha$, then $Q$-a.s.,
\begin{equation}\label{eq:divergent-1}
\int_1^{\infty}e^{-y^{\delta}+V(y)}\,\d y=\infty, \quad \int_1^{\infty}e^{-y^{\delta}-V(y)}\,\d y=\infty.
\end{equation}
On the other hand, if $\delta>1/\alpha$, then $Q$-a.s.,
\begin{equation}\label{eq:convergent-1}
\int_1^{\infty}e^{-y^{\delta}+ V(y)}\,\d y<\infty, \quad \int_1^{\infty}e^{-y^{\delta}-V(y)}\,\d y<\infty.
\end{equation}

\item[{\rm (iii)}] For any $\delta>0$, $Q$-a.s.,
\[
\int_1^{\infty}e^{y^{\delta}+V(y)}\,\d y=\infty, \quad \int_1^{\infty}e^{y^{\delta}-V(y)}\,\d y=\infty.
\]
\end{enumerate}
\end{lem}

\begin{proof}
Since (i) follows by \cite[Proposition 2.1]{KTT17},
we need to prove (ii) and (iii) only.
Note that, if $r_0=1$, then $\{V(x)\}_{x\ge 0}$ is a symmetric $\alpha$-stable process
and so the argument below holds with $r_0$ replaced by any positive constant.
Taking this into account, we assume that $r_0>1$.

We first prove \eqref{eq:divergent-1} in (ii).
We show the former equality in \eqref{eq:divergent-1} only
because the argument for the latter is similar.
Since $r_0>1$, we have
\begin{equation}\label{eq:int-1}
\begin{split}
\int_1^{\infty}e^{-y^{\delta}+V(y)}\,\d y
=\sum_{n=0}^{\infty}\int_{r_0^n}^{r_0^{n+1}} e^{-y^{\delta}+V(y)}\,\d y
&=\sum_{n=0}^{\infty}r_0^n \int_{1}^{r_0} e^{-r_0^{n\delta}z^{\delta}+V(r_0^n z)}\,\d z\\
&=\sum_{n=0}^{\infty}r_0^n \int_{1}^{r_0} e^{-r_0^{n\delta}z^{\delta}+r_0^{n/\alpha}T_{r_0}^nV(z)}\,\d z.
\end{split}
\end{equation}
Then for any $n\ge 1$,
\[
\int_{1}^{r_0} e^{-r_0^{n\delta} z^{\delta}+r_0^{n/\alpha}T_{r_0}^nV(z)}\,\d z
\ge \int_{1}^{r_0} e^{-r_0^{n\delta} z^{\delta}+r_0^{n/\alpha}m(n)}\,\d z
=e^{r_0^{n/\alpha}m(n)}\int_{1}^{r_0} e^{-r_0^{n\delta} z^{\delta}}\,\d z.
\]
For any $n\ge 1$ and $c>0$, since
\[
e^{r_0^{n/\alpha}m(n)}
\ge e^{cr_0^{n/\alpha}}{\bf 1}_{(c,\infty)}(m(n))
\]
and
\[
\int_{1}^{r_0} e^{-r_0^{n\delta} z^{\delta}}\,\d z
\ge e^{-r_0^{(n+1)\delta}}(r_0-1),
\]
we obtain
\begin{equation*}
\begin{split}
e^{r_0^{n/\alpha}m(n)}\int_{1}^{r_0} e^{-r_0^{n\delta} z^{\delta}}\,\d z
&\ge e^{cr_0^{n/\alpha}}{\bf 1}_{(c,\infty)}(m(n))e^{-r_0^{(n+1)\delta}}(r_0-1)\\
&=e^{cr_0^{n/\alpha}-r_0^{(n+1)\delta}}(r_0-1){\bf 1}_{(c,\infty)}(m(n))
\end{split}
\end{equation*}
and so
\[
\int_{1}^{r_0} e^{-r_0^{n\delta}z^{\delta}+r_0^{n/\alpha}T_{r_0}^nV(z)}\,\d z
\ge e^{cr_0^{n/\alpha}-r_0^{(n+1)\delta}}(r_0-1){\bf 1}_{(c,\infty)}(m(n)).
\]
Therefore, it follows by \eqref{eq:int-1} that
\begin{equation}\label{eq:int-2}
\int_1^{\infty}e^{-y^{\delta}+V(y)}\,\d z
\ge (r_0-1)\sum_{n=0}^{\infty}r_0^n e^{cr_0^{n/\alpha}-r_0^{(n+1)\delta}}{\bf 1}_{(c,\infty)}(m(n)).
\end{equation}

By (i) and the proof of \cite[Theorem 1.1]{KTT17},
we have for any $c>0$,
\begin{equation}\label{eq:min-div}
\lim_{N\rightarrow\infty}\sum_{n=1}^N {\bf 1}_{(c,\infty)}(m(n))=\infty, \quad \text{$Q$-a.s.}
\end{equation}
Moreover, if $0<\delta \le 1/\alpha$ and $c>r_0$,
then there exists $N_0\ge 1$ such that for any $n\ge N_0$, $cr_0^{n/\alpha}\ge r_0^{(n+1)\delta}$.
Since $r_0>1$, we have by \eqref{eq:int-2},
\[
\int_1^{\infty}e^{-y^{\delta}+V(y)}\,\d y
\ge (r_0-1)\sum_{n=N_0}^{\infty}{\bf 1}_{(c,\infty)}(m(n))=\infty, \quad \text{$Q$-a.s.}
\]
This yields the former equality in \eqref{eq:divergent-1}.

We next prove \eqref{eq:convergent-1} in (ii).
We show the former equality in \eqref{eq:convergent-1} only
because the argument for the latter is similar.
Assume that $\delta>1/\alpha$.
Then by Lemma \ref{lem:upper},
we have for any $q>1/\alpha$,
\[
Q\left(\text{there exists $R_0>0$ such that for all $x\ge R_0$, $V(x)\le x^{1/\alpha}(\log x)^q$}\right)=1.
\]
Then on the event above, we get
\[
\int_{R_0}^{\infty}e^{-y^{\delta}+V(y)}\,\d y \le \int_{R_0}^{\infty}e^{-y^{\delta}+y^{1/\alpha}(\log y)^{q}}\,\d y<\infty,
\]
which yields the former equality in \eqref{eq:convergent-1}.

We finally prove (iii).
Following the argument for \eqref{eq:int-2}, we have for any $c>0$,
\begin{equation*}
\begin{split}
\int_1^{\infty}e^{y^{\delta}+V(y)}\,\d y
\ge \int_1^{\infty}e^{V(y)}\,\d y
&\ge (r_0-1)\sum_{n=0}^{\infty}r_0^n e^{cr_0^{n/\alpha}}{\bf 1}_{(c,\infty)}(m(n))\\
&\ge (r_0-1)\sum_{n=0}^{\infty}{\bf 1}_{(c,\infty)}(m(n)).
\end{split}
\end{equation*}
By \eqref{eq:min-div}, the last expression above is divergent $Q$-a.s.
Since the similar argument yields $\int_1^{\infty}e^{y^{\delta}-V(y)}\,\d y=\infty$, $Q$-a.s.,
the proof is complete.
\end{proof}

Recall that for a locally bounded function $w:\R \to \R$,
$\{P_t^w\}_{t>0}$ is the Markovian semigroup on $L^2(\R;\mu^w)$
generated by $L^w$.
The next proposition shows that the random drift perturbation can change
the compactness property for diffusion operators $L^{w_1}$ and $L^{w_2}$.

\begin{prop}\label{prop:drift}
Suppose that $d=1$.
For each $k=1,2$, if $\delta\in (0,1/\alpha]$, then $\lambda_{{\rm ess}}^{w_k+W}=0$, $Q$-a.s.
In particular, $\{P_t^{w_k+W}\}_{t>0}$ is noncompact on $L^2({\mathbb R};\mu^{w_k+W})$, $Q$-a.s.
\end{prop}

\begin{proof}
We prove the assertion only for $k=1$ because the argument is similar for $k=2$.
Fix $\delta\in (0,1/\alpha]$.
Then the
generator $L^{w_1+W}$ has the following expression:
\[
L^{w_1+V}=\frac{1}{2}\frac{{\rm d}}{{\rm d}\mu^{w_1+W}}\frac{{\rm d}}{{\rm d}s^{w_1+W}},
\quad \mu^{w_1+W}(\d x)=e^{-(w_1(x)+W(x))}\,\d x, \quad s^{w_1+W}(x)
=\int_0^x e^{w_1(y)+W(y)}\,{\rm d}y.
\]

Lemma \ref{lem:divergent} (ii), (iii) imply that $Q$-a.s., for any $x>0$,
\[
\int_x^{\infty}e^{w_1(y)+V(y)}\,\d y=\int_x^{\infty}e^{y^{\delta}+V(y)}\,\d y=\infty, \quad
\int_x^{\infty}e^{-(w_1(y)+V(y))}\,\d y=\int_x^{\infty}e^{-y^{\delta}-V(y)}\,\d y=\infty.
\]
Hence, we have  $Q$-a.s., for any $x>0$,
\begin{equation}\label{eq:product}
\int_0^x e^{w_1(y)+V(y)}\,\d y\cdot \int_x^{\infty}e^{-(w_1(y)+V(y))}\,\d y=\infty.
\end{equation}
Then by \cite[Theorem 2]{P09}, we obtain the desired assertion.
\end{proof}

We close this section with the following remarks.
\begin{rem}\rm
\begin{enumerate}
\item[(i)]
Assume
first
that $\alpha\in (0,1)$ and so $1/\alpha>1$.
Then for any $\delta\in (1,1/\alpha]$ and $k=1,2$,
$L^{w_k}$ generates a compact semigroup on $L^2({\mathbb R};\mu^{w_k})$,
but Proposition \ref{prop:drift} says that $L^{w_k+W}$ generates a noncompact semigroup on $L^2({\mathbb R};\mu^{w_k+W})$, $Q$-a.s.
Namely, the random drift $W$ can change the compactness property of the semigroup.
Assume
next
that $\alpha=1$.
If we take $\delta=1$, then for each $k=1,2$, $\lambda_{{\rm ess}}^{w_k}\in (0,\infty)$,
but Proposition \ref{prop:drift} yields $\lambda_{{\rm ess}}^{w_k+W}=0$, $Q$-a.s.
Therefore, the random drift $W$ can change the positivity of the bottom of the essential spectrum.

\item[(ii)]
For $\delta>1/\alpha$, we do not know if $L^{w_1+W}$ generates a compact Markovian semigroup on $L^2(\R;\mu^{w_1+W})$.
In fact, by Lemma \ref{lem:divergent} (ii), we have $Q$-a.s., for any $x>0$, $\int_x^{\infty}e^{-(y^{\delta}+V(y))}\,\d y<\infty$.
However, since we have no information about the decay rate of $\int_x^{\infty}e^{-(y^{\delta}+V(y))}\,\d y$ as $x\rightarrow\infty$,
it is unclear if the left hand side of \eqref{eq:product} is divergent or not as $x\rightarrow\infty$.
Therefore, we could not apply \cite[Theorem 2]{P09} to $L^{w_1+W}$.
On the other hand, for $\delta>1/\alpha$, we do not know if $L^{w_2+W}$ generates a compact Markovian semigroup on $L^2(\R;\mu^{w_2+W})$.
In fact, by Lemma \ref{lem:divergent} (ii),
we have  $Q$-a.s., for any $x>0$,
$
\int_x^{\infty}e^{-y^{\delta}+V(y)}\,\d y<\infty
$.
Define $h(x)=\int_x^{\infty}e^{-y^{\delta}+V(y)}\,\d y$ for $x>0$.
In order to apply \cite[Theorem 2]{P09} to $L^{w_2+W}$,
we need to calculate
\[
\limsup_{x\rightarrow\infty}\left(\frac{1}{h(x)}-\frac{1}{h(0)}\right)\int_x^{\infty}h(y)^2 e^{y^{\delta}-V(y)}\,\d y,
\]
but this value is unknown at present.
\end{enumerate}
\end{rem}

\noindent
{\bf Acknowledgements.}
The research of Yuichi Shiozawa is supported by JSPS KAKENHI Grant Number JP23K25773.
The research of Jian Wang is supported by the NSF of China the National Key R\&D Program of China (2022YFA1006003) and the National Natural Science
Foundation of China (Nos. 12225104 and 12531007).

\end{document}